\documentclass[a4paper,12pt]{article}

\usepackage[utf8]{inputenc}
\usepackage[english]{babel}
\usepackage{enumitem}
\usepackage{amsmath}
\usepackage{amssymb}
\usepackage{amsthm}
\usepackage{amsfonts}

\usepackage{caption}
\newtheorem{dfn}{Definition}[section]
\newtheorem{lem}[dfn]{Lemma}
\newtheorem{thm}[dfn]{Theorem}

\newtheorem{alg}[dfn]{Algorithm}

\theoremstyle{definition}

\newtheorem{asm}[dfn]{Assumption}
\newtheorem{exm}[dfn]{Example}
\newtheorem{rem}[dfn]{Remark}

\allowdisplaybreaks

\title{Convergence Results of Forward-Backward Algorithms for Sum of Monotone Operators in Banach Spaces}

\date{January 22, 2021}

\author{Yekini
Shehu\footnote{Department of Mathematics, Zhejiang Normal University, Jinhua,
321004, People's Republic of China; Institute of Science and Technology (IST), Am Campus 1, 3400, Klosterneuburg, Vienna, Austria; e-mail: yekini.shehu@unn.edu.ng. }}

\begin{document}

\maketitle

\begin{abstract}
\noindent
It is well known that many problems in image recovery, signal processing, and machine learning can be modeled as finding zeros of the sum of maximal monotone and Lipschitz continuous monotone operators. Many papers have studied forward-backward splitting methods for finding zeros of the sum of two monotone operators in Hilbert spaces. Most of the proposed splitting methods in the literature have been proposed for the sum of maximal monotone and inverse-strongly monotone operators in Hilbert spaces. In this paper, we consider splitting methods for finding zeros of the sum of maximal monotone operators and Lipschitz continuous monotone operators in Banach spaces. We obtain weak and strong convergence results for the zeros of the sum of maximal monotone and Lipschitz continuous monotone operators in Banach spaces. Many already studied problems in the literature can be considered as special cases of this paper.\\



\end{abstract}

\section{Introduction}
\noindent Let $E$ be a real Banach space with norm $\|.\|,$ we denote by $E^*$ the dual of $E$ and $\langle f,x\rangle$ the value of $f\in E^*$ at $x\in E.$ Let $B:E\rightarrow 2^{E^*}$ be a maximal monotone operator and $A:E\rightarrow E^*$ be a Lipschitz continuous monotone operator. We consider the following inclusion problem: find $x\in E$ such that
\begin{eqnarray}\label{bami2}
   0 \in (A+B)x.
\end{eqnarray}
Throughout this paper, we denote the solution set of the inclusion problem (\ref{bami2}) by $(A+B)^{-1}(0)$. \\

\noindent
The inclusion problem \eqref{bami2} contains, as special cases, convexly constrained linear inverse problem, split feasibility
problem, convexly constrained minimization problem, fixed point problems, variational inequalities, Nash
equilibrium problem in noncooperative games, and many more. See, for instance, \cite{CChen,Combettes10,LionsMercier,MoudafiThera,Passtygb,PeacemanRachford} and the
references therein.\\

\noindent A popular method for solving problem \eqref{bami2} in real Hilbert spaces, is the well-known forward–backward splitting method introduced by Passty \cite{Passtygb} and Lions and Mercier \cite{LionsMercier}. The method is formulated as
\begin{eqnarray}\label{tseng}
 x_{n+1}=(I+\lambda_n B)^{-1}(I-\lambda_n A)x_n, ~~\lambda_n>0,
\end{eqnarray}
under the condition that $Dom(B) \subset Dom(A)$.
It was shown, see for example \cite{CChen}, that weak convergence of \eqref{tseng} requires quite restrictive
assumptions on $A$ and $B$, such that the inverse of $A$ is strongly monotone or $B$
is Lipschitz continuous and monotone and the operator $A + B$ is strongly monotone
on $Dom(B)$. Tseng in \cite{Tseng}, weakened these assumptions and included an
extra step per each step of \eqref{tseng} (called Tseng's splitting algorithm) and obtained weak convergence result in real Hilbert spaces. Quite recently, Gibali and Thong \cite{GibaliThong} have obtained strong convergence result by modifying
Tseng's splitting algorithm in real Hilbert spaces.\\

\noindent In this paper, we extend Tseng's result \cite{Tseng} to a Banach
space. We first prove the weak convergence of the sequence generated by our proposed method, assuming
that the duality mapping is weakly sequentially continuous. This weak convergence is
a generalization of Theorem 3.4 given in \cite{Tseng}. We next prove the strong convergence result for problem \eqref{bami2} under some mild assumptions and this extends Theorems 1 and 2 in \cite{GibaliThong} to Banach spaces. Finally, we apply our convergence results to the composite convex minimization problem in Banach spaces.

\section{Preliminaries}\label{Sec:Prelims}
In this section, we define some concepts and state few basic results that we will
 use in our subsequent analysis.
Let $S_E$ be the unit sphere of $E$, and $B_E$ the closed unit ball of $E$.\\

\noindent
Let $\rho_E:[0,\infty)\rightarrow [0,\infty)$ be the modulus of
smoothness of $E$ defined by
$$
\rho_E(t):=\sup\Big\{\frac{1}{2}(\|x+y\|+\|x-y\|)-1:\,x\in
S_E,\,\|y\|\leq t\Big\}.
$$
\noindent A Banach space $E$ is said to be $2$-uniformly smooth, if there exists a fixed
constant $c>0$ such that $\rho_E(t)\leq ct^2$. The space $E$ is said to be \textit{smooth} if
\begin{eqnarray}\label{ori}
\lim_{t\rightarrow 0} \frac{\|x+ty\|-\|x\|}{t}
\end{eqnarray}
exists for all $x,y\in S_E$. The space $E$ is also said to be \textit{uniformly smooth} if (\ref{ori}) converges
uniformly in $x,y\in S_E$. It is well known that if $E$ is $2$-uniformly smooth, then $E$ is uniformly smooth. It is said to be \textit{strictly convex} if $\|(x + y)/2\| <1$ whenever $x,y\in S_E$
and $x\neq y$. It is said to be \textit{uniformly convex} if $\delta_E(\epsilon) >0$ for all $\epsilon \in (0,2]$, where $ \delta_E$ is the \textit{modulus of convexity} of $E$ defined by
\begin{eqnarray}\label{ori2}
\delta_E(\epsilon):=\inf \Big\{1-\Big|\Big|\frac{x+y}{2}\Big|\Big|\mid x,y\in B_E, \|x-y\|\geq \epsilon \Big\}
\end{eqnarray}
for all $\epsilon \in [0,2]$. The space $E$ is said to be \textit{2-uniformly convex} if there exists $c >0$ such
that $\delta_E(\epsilon)\geq c\epsilon^2$ for all $\epsilon \in [0,2]$.
It is obvious that every 2-uniformly convex Banach
space is uniformly convex.
It is known that all Hilbert spaces are uniformly smooth and 2-uniformly convex.
It is also known that all the Lebesgue spaces $L_p$ are uniformly smooth and 2-uniformly convex whenever $1<p\leq 2$ (see \cite{Beauzamy}).\\

\noindent
 The \textit{normalized duality mapping} of $E$ into $E^*$ is defined by
$$
Jx:=\{x^* \in E^*\mid\langle x^*,x\rangle=\|x^*\|^2=\|x\|^2\}
$$
for all $x\in E$. The normalized duality mapping $J$ has the following properties (see, e.g., \cite{Taka}):
\begin{itemize}
	\item if $E$ is reflexive and strictly convex with the strictly convex dual space $E^*$, then $J$ is single-valued, one-to-one and onto mapping. In this case, we can define the single-valued mapping $J^{-1}:E^*\rightarrow E$ and we have
$J^{-1}=J_{*}$, where $J_{*}$ is the normalized duality mapping on $E^*$;
	\item if $E$ is uniformly smooth, then $J$ is uniformly norm-to-norm continuous on each bounded subset of $E.$
\end{itemize}
Let us recall from \cite{alber1,Cioranescu} some examples for the normalized duality mapping $J$ in the uniformly convex and uniformly smooth Banach spaces $\ell_p$ and $L_p, 1 < p <\infty$.
\begin{itemize}
  \item For $\ell_p: Jx=\|x\|_{\ell_p}^{2-p}y \in \ell_q$, where $ x=(x_j)_{j\geq 1}$ and $y=(x_j|x_j|^{p-2})_{j\geq 1}$, $\frac{1}{p}+\frac{1}{q}=1$.
  \item For $L_p: Jx=\|x\|_{L_p}^{2-p}|x|^{p-2}x \in L_q$,  $\frac{1}{p}+\frac{1}{q}=1$.
\end{itemize}

\noindent Now, we recall some fundamental and useful results.

\begin{lem}
The space $E$ is 2-uniformly convex if and only if there exists $\mu_E \geq 1$
such that
\begin{eqnarray}\label{opari}
\frac{\|x+y\|^2+\|x-y\|^2}{2}\geq \|x\|^2+\|\mu^{-1}_E y\|^2
\end{eqnarray}
for all $x,y\in E$.
\end{lem}
\noindent
The minimum value of the set of all $\mu_E \geq 1$ satisfying (\ref{opari}) for all $x,y\in E$ is denoted by
$\mu$ and is called the \textit{2-uniform convexity constant} of $E$; see \cite{Ball}. It is obvious that $\mu=1$
whenever $E$ is a Hilbert space.

\begin{lem}[\cite{Avetisyan}]\label{lemAv}
Let $\displaystyle\frac{1}{p}+\frac{1}{q}=1,~~p,q>1$. The space $E$ is $q-$uniformly smooth if and only if its dual $E^*$ is $p-$uniformly convex.
\end{lem}

\begin{lem}[\cite{Xu}]\label{lemXu}
Let $E$ be a real Banach space. The following are equivalent:
\begin{itemize}
\item[{\rm (1)}] $E$ is 2-uniformly smooth
\item[{\rm (2)}] There exists a constant $\kappa>0$ such that $\forall\ x,y\in E$,
$$\|x+y\|^2\leq \|x\|^2+2\langle y,J(x)\rangle+2\kappa^2\|y\|^2,$$
where $\kappa$ is the 2-uniform smoothness constant. In Hilbert spaces, $\kappa=\frac{1}{\sqrt{2}}$.
\end{itemize}
\end{lem}

\begin{dfn}\label{Def:LipMon}
Let $ X \subseteq E $ be a nonempty subset. Then a mapping $ A: X \to E^* $
is called
\begin{itemize}
   \item[(a)] {\em strongly monotone} with modulus $\gamma>0$ on $ X $ if
$$\langle Ax-Ay, x-y\rangle \geq \gamma\|x-y\|^2, \forall x,y \in X.$$
\noindent In this case, we say that $A$ is $\gamma$-strongly monotone;
\item[(b)]  {\em monotone} on $ X $ if
$$\langle Ax-Ay, x-y\rangle \geq 0, \forall x,y \in X;$$
 \item[(c)] {\em Lipschitz continuous} on $ X $ if there exists a
      constant $ L > 0 $ such that\\ $ \| Ax - Ay \| \leq L \| x-y \| $
      for all $ x, y \in X $.
\end{itemize}
\end{dfn}

\noindent We give some examples of monotone operator in Banach spaces as given in \cite{AlberRyazantseva}.

\begin{exm}
Let $G\subset \mathbb{R}^n$ be a bounded measurable domain.  Define the operator $A:L^p(G)\rightarrow L^q(G),~~\frac{1}{p}+\frac{1}{q} = 1, ~~p>1$, by the formula
$$
Ay(x):=\varphi(x,|y(x)|^{p-1})|y(x)|^{p-2}y(x),~~x \in G,
$$
\noindent where the function $\varphi(x,s)$ is measurable as a function of $x$ for every $s\in [0,\infty)$ and
continuous for almost all $x\in G$ as a function on $s, |\varphi(x,s)|\leq M$ for all $s\in [0,\infty)$
and for almost all $x\in G$. Observe that the operator $A$ really maps $L^p(G)$ to $L^q(G)$ because of
the inequality $|Ay|\leq M|y|^{p-1}$. Then it can be shown that $A$ is a monotone map on $L^p(G)$.
\end{exm}
\noindent
Let us consider another example from quantum mechanics.
\begin{exm}
Define the operator
$$
Au:=-a^2\triangle u+(g(x)+b)u(x)+u(x)\int_{\mathbb{R}^3} \frac{u^2(y)}{|x-y|}dy,
$$
\noindent where $\triangle:=\sum_{i=1}^3 \frac{\partial^2}{\partial x_i^2}$ is the Laplacian in $\mathbb{R}^3$, $a$ and $b$ are constants, $g(x)=g_0(x)+g_1(x),~~g_0(x) \in L^{\infty}(\mathbb{R}^3), g_1(x) \in L^2(\mathbb{R}^3)$. Let $A:= L+B$, where
the operator $L$ is the linear part of $A$ (it is the Schr\"odinger operator) and $B$ is defined by
the last term. It is known that $B$ is a monotone operator on $L^2(\mathbb{R}^3)$ (see page 23 of \cite{AlberRyazantseva}) and this implies that $A:L^2(\mathbb{R}^3)\rightarrow L^2(\mathbb{R}^3)$ is also a monotone operator.
\end{exm}

\begin{exm}
\noindent This example gives one of the perhaps most famous example of monotone operators, viz.
the $p$-Laplacian $-{\rm div}(|\nabla u|^{p-2}\nabla u): W^1_0(L_p(\Omega))\rightarrow \Big (W^1_0(L_p(\Omega))\Big)^*$, where $u:\Omega \rightarrow \mathbb{R}$ is a real function defined on a domain $\Omega \subset \mathbb{R}^n$. The $p$-Laplacian operator is a monotone operator for $1<p<\infty$ (in fact, it is strongly monotone for $p \geq 2$, and strictly monotone for $1 < p < 2$). The $p$-Laplacian operator is an extremely important model in many topical applications and certainly played an important role in the development of the theory of monotone operators.
\end{exm}

\begin{dfn}\label{Def:LipMon1}
A multi-valued operator $B:E \rightarrow 2^{E^*}$ with graph $G(T)= \{(x,x^*): x^* \in Tx\}$ is said to be monotone if for any $x,y \in
D(T),x^* \in Tx$ and $y^* \in Ty$
$$
\langle x-y,x^*-y^*\rangle \geq 0.
$$
\noindent A monotone operator $B$ is said to be maximal if $B =S$ whenever $S:E \rightarrow 2^{E^*}$ is monotone and $G(B)\subset G(S)$.
\end{dfn}
\noindent Let $E$ be a reflexive, strictly convex and smooth Banach space and let $B:E \rightarrow 2^{E^*}$ be a maximal monotone operator.
Then for each $r > 0$ and $x \in E$, there corresponds a unique element $x_r \in E$ such that
$$
Jx\in Jx_r+rBx_r.
$$
\noindent We define this unique element $x_r$, the resolvent of $B$, denoted  by $J^B_rx$. In other
words, $J_r^B=(J +rB)^{-1}J$ for all $r > 0$. It is easy to show that $B^{-1}0 = F(J^B_r)$ for
all $r > 0$, where $F(J^B_r)$ denotes the set of all fixed points of $J^B_r$. We can also define,
for each $r > 0$, the Yosida approximation of $B$ by $A_r=\frac{J-JJ^B_r}{r}$. For more details, see, for instance
\cite{Barbu}.\\

\noindent Suppose $E$ is a smooth Banach space. We introduce the functional studied in \cite{alber1,KamimuraTakahashi,Reichbook}: $\phi:E \times E\rightarrow \mathbb{R}$ defined by:
\begin{eqnarray}\label{new}
\phi(x,y):=\|x\|^2-2\langle x,Jy\rangle +\|y\|^2.
\end{eqnarray}
Clearly,
$$\phi(x,y)\geq (\|x\|-\|y\|)^2 \geq 0.$$
\noindent
The following lemma gives some identities of functional $\phi$ defined in \eqref{new}.

\begin{lem}\label{lm29} (see \cite{Aoyama} and \cite{alber1})
Let $E$ be a real uniformly convex, smooth Banach space.
Then, the following identities hold:\\
(i)
\begin{align*}
\phi(x,y)= \phi(x,z)+ \phi(z,y)+2\langle x-z, Jz-Jy\rangle, \ \forall x,y,z\in E.
\end{align*}
(ii)
\begin{align*}
\phi(x,y)+\phi(y,x) = 2\langle x-y, Jx-Jy\rangle, \ \forall x,y\in E.
\end{align*}
\end{lem}
\noindent Let $C \subseteq E $ be a nonempty, closed and convex subset of a real,
uniformly convex Banach space $ E $. Let us introduce the functional $V(x,y):E \times E^{*}\rightarrow \mathbb{R}$ by the formula:
\begin{equation}\label{neww}
V(x,y):=\|x\|^2_E-2\langle x,y\rangle +\|y\|^2_{E^*}.
\end{equation}
Then, it is easy to see that
$$V(y,x)=\phi(y,J^{-1}x),~~\forall x \in E^*,y \in E.$$

\noindent
In the next lemma, we describe the property of the operator $V(.,.)$ defined in \eqref{neww}.

\begin{lem}(\cite{alber1})\label{new2}
$$V(x,x^*)+2\langle J^{-1}x^*-x,y^*\rangle \leq V(x,x^*+y^*),~~\forall x \in E,~~x^*,y^*\in E^*.$$
\end{lem}

\noindent
The lemma that follows is stated and proven in \cite[Lem.\ 2.2]{Aoyama}.
\begin{lem}\label{aoy}
Suppose that $E$ is 2-uniformly convex Banach space.
Then, there exists $\mu \geq 1$ such that
$$\frac{1}{\mu}\|x-y\|^2\leq \phi(x,y)~~\forall x,y\in E.$$
\end{lem}

\noindent The following lemma was given in \cite{Iiduka}.

\begin{lem}\label{lm28}
Let $S$ be a nonempty, closed convex subset of a uniformly convex, smooth Banach space $E$. Let $\{x_n\}$ be
a sequence in $E$. Suppose that, for all $u\in S$,
$$
\phi(u,x_{n+1}) \leq \phi(u,x_n),~~\forall n \geq 1.
$$
\noindent Then $\{\Pi_S(x_n)\}$ is a Cauchy sequence.
\end{lem}

\noindent The following property of $\phi(.,.)$ was given in \cite[Thm.\ 7.5]{alber1} (see also \cite{Diestel,Figiel}).
\begin{lem}\label{alber}
Let $E$ be a uniformly smooth Banach space which is also uniformly convex. If $\|x\| \leq c, \|y\|\leq c$, then
$$2L_1^{-1}c^2\delta_E\Big(\frac{\|x-y\|}{4c}\Big)\leq \phi(y,x)\leq 4L_1^{-1}c^2\rho_E\Big(\frac{4\|x-y\|}{c}\Big),$$
where $L_1 (1 < L_1 < 3.18)$ is the Figiel's constant.
\end{lem}

\noindent
We next recall some existing results from the literature to facilitate
our proof of strong convergence. The first is taken from  \cite{mainge}.

\begin{lem}\label{mm}
Let $\{a_n\}$ be sequence of real numbers such that there exists a subsequence $\{n_i\}$ of $\{n\}$
such that $a_{n_i} < a_{{n_i}+1}$, for all $i\in \mathbb{N}$. Then there exists a nondecreasing sequence $\{m_k\}\subset \mathbb{N}$ such that $m_k\rightarrow \infty$
and the following properties are satisfied by all (sufficiently large) numbers $k\in \mathbb{N}$
$$a_{m_k} \leq a_{{m_k}+1}~~{\rm and}~~a_k\leq a_{{m_k}+1}.$$

In fact, $m_k=\max\{j\leq k : a_j<a_{j+1}\}$.
\end{lem}

\begin{lem}(\cite{XU1})\label{lm23}
Let $\{a_n\}$ be a sequence of nonnegative real numbers satisfying
the following relation:
$$
   a_{n+1}\leq(1-\alpha_n)a_n+\alpha_n\sigma_n+\gamma_n,~~n \geq 1,
$$
where
\begin{itemize}
   \item[(a)] $\{\alpha_n\}\subset[0,1],$ $ \sum_{n=1}^{\infty} \alpha_n=\infty;$
   \item[(b)] $\limsup\sigma_n \leq 0$;
   \item[(c)] $\gamma_n \geq 0 \ (n \geq 1),$ $ \sum_{n=1}^{\infty}
      \gamma_n <\infty.$
\end{itemize}
Then, $a_n\rightarrow 0$ as $n\rightarrow \infty$.
\end{lem}

\noindent
The following lemma is needed in our proof to show that the weak limit point is a solution to the inclusion problem \eqref{bami2}.

\begin{lem} (\cite{Barbu})\label{babu2}
Let $B:E \to 2^{E^*} $ be a maximal monotone mapping and $ A: E \to E^* $ be a Lipschitz continuous and monotone mapping.
 Then the mapping
$A+B$ is a maximal monotone mapping.
\end{lem}

\noindent The following result gives an equivalence of fixed point problem and problem \eqref{bami2}.

\begin{lem}\label{babu}
Let $B:E \to 2^{E^*} $ be a maximal monotone mapping and
$ A: E \to E^* $ be a mapping. Define a mapping
$$
T_\lambda x:= J_{\lambda}^BoJ^{-1}(J - \lambda A)(x),~~x \in E, \lambda>0.
$$
\noindent Then $F(T_\lambda)=(A+B)^{-1}(0),$ where $F(T_\lambda)$ denotes the set of all fixed points of $T_\lambda$.
\end{lem}

\begin{proof}
Let $x \in   F(T_\lambda)$. Then

\begin{eqnarray*}
 x \in   F(T_\lambda) &\Leftrightarrow& x=T_\lambda x=J_{\lambda}^BoJ^{-1}(J - \lambda A)(x) \\
 &\Leftrightarrow& x = (J+\lambda B)^{-1} JoJ^{-1}(Jx - \lambda Ax) \\
 &\Leftrightarrow& Jx - \lambda Ax \in Jx+\lambda Bx \\
  &\Leftrightarrow& 0 \in \lambda (Ax+Bx)\\
  &\Leftrightarrow& 0 \in Ax+Bx \\
  &\Leftrightarrow& x \in (A+B)^{-1}(0).
\end{eqnarray*}

\end{proof}

\noindent We shall adopt the following notation in this paper:\\
. $x_n\rightarrow x$ means that $x_n\rightarrow x$ strongly.\\
. $x_n\rightharpoonup x$ means that $x_n\rightarrow x$ weakly.\\

\section{Approximation Method}\label{Sec:Method}
\noindent
In this section, we propose our method and state certain conditions under which we obtain the desired convergence for our proposed methods. First, we give the conditions governing the cost function and the sequence of parameters below.

\begin{asm}\label{Ass:VI}
\begin{itemize}
   \item[{\rm  (a)}] Let $E$ be a real 2-uniformly convex Banach space which is also uniformly smooth.
   \item[{\rm (b)}] Let $B:E \to 2^{E^*} $ be a maximal monotone operator;
      $ A: E \to E^* $ a monotone and $L$-Lipschitz continuous.
   \item[{\rm (c)}] The solution set $ (A+B)^{-1}(0) $ of the inclusion problem \eqref{bami2} is nonempty.
\end{itemize}
\end{asm}

\noindent
Throughout this paper, we assume that the duality mapping $J$ and the resolvent $J_{\lambda_n}^B:=(J+\lambda_nB)^{-1} J$ of maximal monotone operator $B$ are easy to compute.

\begin{asm}\label{Ass:Parameters}
Suppose the sequence $ \{ \lambda_n \}_{n=1}^\infty $ of step-sizes satisfies the following condition:
$$0<a\leq\lambda_n\leq b<\displaystyle\frac{1}{\sqrt{2\mu}\kappa L}$$
\noindent
where
 \begin{itemize}
                \item[] $\mu$ is the 2-uniform convexity constant of $E$;
                \item[] $\kappa$ is the 2-uniform smoothness constant of $E^*$;
                \item[] $L$ is the Lipschitz constant of $A$.
     \end{itemize}
\end{asm}

\noindent Assumption \ref{Ass:Parameters} is satisfied, e.g., for $ \lambda_n = a+\frac{n}{n+1}\Big(\frac{1}{\sqrt{2\mu}\kappa L}-a\Big)$ for all
$ n \geq 1$.

\noindent We now give our proposed method below.

\begin{alg}\label{Alg:AlgL}$\left.  {}\right.$\newline
\textbf{Step 0}: Let Assumptions~\ref{Ass:VI} and \ref{Ass:Parameters} hold. Let $ x_1 \in E $ be a given starting point. Set $ n := 1 $.\newline

\noindent\textbf{Step 1}: Compute $y_n:=J_{\lambda_n}^BoJ^{-1}(Jx_n - \lambda_nAx_n)$. If $x_n-y_n=0$: STOP.\newline

\noindent\textbf{Step 2}: Compute
\begin{eqnarray}\label{e31}
 x_{n+1} = J^{-1}[Jy_n - \lambda_n(Ay_n-Ax_n)].
\end{eqnarray}

\noindent\textbf{Step 3}: Set $ n \leftarrow n+1 $, and go to \textbf{Step 1}.
\end{alg}

\noindent
We observe that in real Hilbert spaces, the duality mapping $J$ becomes the identity mapping and our Algorithm \ref{Alg:AlgL} reduces to the algorithm proposed by Tseng in \cite{Tseng}. \\

\noindent
Note that both sequences $\{y_n\}$ and $\{x_n\}$ are in $E$. Furthermore, by Lemma \ref{babu}, we have that if $x_n=y_n$, then $x_n$ is a solution of problem \eqref{bami2}.\\

\noindent
To the best of our knowledge, the proposed Algorithm \ref{Alg:AlgL} is the only known algorithm which can solve monotone inclusion problem
\eqref{bami2} without the inverse-strongly monotonicity of $A$. We consider some various cases of Algorithm \ref{Alg:AlgL}.

\begin{itemize}
  \item When $A=0$ in Algorithm \ref{Alg:AlgL}, then Algorithm \ref{Alg:AlgL} reduces to the methods proposed in \cite{Barbu,Guler,Kamimura, Kohsaka,Lionspl,LionsMercier,Martinet22,Passtygb,Reichbook,Rockafellar27,SolodovSvaiter}.
  In this case, the assumption that $E$ is 2-uniformly convex Banach space and uniformly smooth is not needed. In fact, the convergence can be obtained in reflexive Banach spaces in this case. However, we do not know if the convergence of Algorithm \ref{Alg:AlgL} can be obtained in a more general reflexive Banach space for problem \eqref{bami2}.

  \item When $B=N_C$, the normal cone for closed and convex subset $C$ of $E$ ($N_C(x):=\{x^* \in E^*:\langle y-x,x^*\rangle \leq 0, \forall y \in C \}$), then the inclusion problem \eqref{bami2} reduces to a variational inequality problem (i.e., find $x \in C: \langle Ax,y-x\rangle \geq 0,~\forall y \in C$). It is well known that $N_C=\partial \delta_C$, where $\delta_C$ is the indicator function of $C$ at $x$, defined by $\delta_C(x)=0$ if $x \in C$ and $\delta_C(x)=+\infty$ if $x \notin C$ and
      $\partial (.)$ is the subdifferential, defined by $\partial f(x):=\{x^* \in E^*: f(y)\geq f(x)+\langle x^*, y-x\rangle,~~\forall y \in E\}$ for a proper, lower semicontinuous convex functional $f$ on $E$. Using the theorem of  Rockafellar in \cite{RockafellarRT1,RockafellarRT2}, $N_C=\partial \delta_C$ is maximal monotone. Hence,
      $$
      Jz \in J (J_{\lambda_n}^B)+\lambda_n \partial \delta_C(J_{\lambda_n}^B),~~\forall z \in E.
      $$
  \noindent This implies that
  \begin{eqnarray*}
  0\in \partial \delta_C(J_{\lambda_n}^B)+\frac{1}{\lambda_n}J (J_{\lambda_n}^B)-\frac{1}{\lambda_n}Jz
 =\partial \Big(\delta_C+\frac{1}{2\lambda_n}\|.\|^2-\frac{1}{\lambda_n}Jz \Big)J_{\lambda_n}^B.
  \end{eqnarray*}
Therefore,
$$
J_{\lambda_n}^B(z)={\rm argmin}_{y \in E}\Big\{\delta_C(y)+\frac{1}{2\lambda_n}\|y\|^2-\frac{1}{\lambda_n}\langle y,Jz\rangle \Big\}
$$
\noindent
and $y_n$ in Algorithm \ref{Alg:AlgL} reduces to
$$
y_n={\rm argmin}_{y \in E}\Big\{\delta_C(y)+\frac{1}{2\lambda_n}\|y\|^2-\frac{1}{\lambda_n}\langle y,Jx_n - \lambda_nAx_n\rangle \Big\}.
$$
\end{itemize}

\noindent
However, in implementing our proposed Algorithm \ref{Alg:AlgL}, we assume that the resolvent $(J+\lambda_nB)^{-1} J$ is easy to compute and the duality mapping $J$ is easily computable as well. On the other hand, one has to obtain the Lipschitz constant, $L$, of the monotone mapping $A$ (or an estimate of it). In a case when the Lipschitz constant cannot be accurately estimated or overestimated, this might result in too small step-sizes $\lambda_n$. This is a drawback of our proposed Algorithm \ref{Alg:AlgL}. One way to overcome this obstacle is to introduce linesearch in our Algorithm \ref{Alg:AlgL}. This case will be considered in Algorithm \ref{Alg:AlgL1}.

\subsection{Convergence Analysis}\label{Sec:Convergence}
\noindent
In this Section, we give the convergence analysis of the proposed Algorithm \ref{Alg:AlgL}. First, we establish the boundedness of the sequence of iterates generated by Algorithm \ref{Alg:AlgL}.

\begin{lem}\label{t31}
Let Assumptions~\ref{Ass:VI} and \ref{Ass:Parameters} hold. Assume that $x^*\in (A+B)^{-1}(0)$ and
let the sequence $ \{x_n\}_{n=1}^\infty $ be generated by Algorithm~\ref{Alg:AlgL}.
Then $\{x_n\}$ is bounded.
\end{lem}

\begin{proof}
By the Lyaponuv functional $\phi$, we have
\begin{align}
\phi(x^*,x_{n+1})=&\phi(x^*,J^{-1}(Jy_n - \lambda_n(Ay_n-Ax_n))) \nonumber \\
=& \|x^*\|^2-2\langle x^*,JJ^{-1}(Jy_n - \lambda_n(Ay_n-Ax_n))\rangle\nonumber\\
&+\|J^{-1}(Jy_n - \lambda_n(Ay_n-Ax_n))\|^2\nonumber\\
=& \|x^*\|^2-2\langle x^*,Jy_n - \lambda_n(Ay_n-Ax_n)\rangle+\|(Jy_n - \lambda_n(Ay_n-Ax_n))\|^2\nonumber\\
=& \|x^*\|^2-2\langle x^*,Jy_n\rangle+2\lambda_n\langle x^*, Ay_n-Ax_n\rangle
\nonumber\\
&+\|Jy_n - \lambda_n(Ay_n-Ax_n)\|^2.\label{m2}
\end{align}
Using Lemma~\ref{lemAv}, we get that $E^*$ is 2-uniformly smooth  and so by Lemma~\ref{lemXu}, we get
\begin{align}
\|Jy_n - \lambda_n(Ay_n-Ax_n)\|^2\leq& \|Jy_n\|^2-2\lambda_n\langle Ay_n-Ax_n,y_n\rangle\nonumber\\
& +2\kappa^2\|\lambda_n(Ay_n-Ax_n)\|^2.\label{me3}
\end{align}
Substituting \eqref{me3} into \eqref{m2}, we get
\begin{align}
\phi(x^*,x_{n+1})\leq&\|Jy_n\|^2 -2 \lambda_n\langle Ay_n-Ax_n,y_n\rangle +2\kappa^2\|\lambda_n(Ay_n-Ax_n)\|^2\nonumber\\
& +\|x^*\|^2-2\langle x^*,Jy_n\rangle +  2 \lambda_n\langle x^*, Ay_n-Ax_n\rangle\nonumber\\
=& \|x^*\|^2-2\langle x^*, Jy_n\rangle + \|y_n\|^2-2 \lambda_n\langle Ay_n-Ax_n,y_n-x^*\rangle\nonumber\\
&+ 2\kappa^2\|\lambda_n(Ay_n-Ax_n)\|^2  \nonumber\\
=& \phi(x^*,y_n) -2 \lambda_n\langle Ay_n-Ax_n,y_n-x^*\rangle + 2\kappa^2\|\lambda_n(Ay_n-Ax_n)\|^2.           \label{me4}
\end{align}

\noindent Using Lemma \ref{lm29} (i), we get
\begin{align}
\phi(x^*,y_n)=& \phi(x^*,x_n) +\phi(x_n,y_n) + 2\langle x^*-x_n, Jx_n-Jy_n\rangle \nonumber\\
=& \phi(x^*,x_n) +\phi(x_n,y_n) + 2\langle x_n-x^*, Jy_n-Jx_n\rangle. \label{me5}
\end{align}

\noindent Putting \eqref{me5} into \eqref{me4}, we get
\begin{align}
\phi(x^*,x_{n+1})=& \phi(x^*,x_n) + \phi(x_n,y_n) + 2\langle x_n-x^*, Jy_n-Jx_n\rangle \nonumber\\
&- 2 \lambda_n\langle Ay_n-Ax_n,y_n-x^*\rangle + 2\kappa^2\|\lambda_n(Ay_n-Ax_n)\|^2 \nonumber \\
=& \phi(x^*,x_n) + \phi(x_n,y_n) -2\langle y_n-x_n, Jy_n-Jx_n\rangle + 2\langle y_n-x^*, Jy_n-Jx_n\rangle \nonumber\\
& - 2 \lambda_n\langle Ay_n-Ax_n,y_n-x^*\rangle + 2\kappa^2\|\lambda_n(Ay_n-Ax_n)\|^2.    \label{me6}
\end{align}
Using Lemma \ref{lm29} (ii), we get
\begin{equation}
- \phi(y_n,x_n) + 2\langle y_n-x_n, Jy_n-Jx_n\rangle = \phi(x_n,y_n).    \label{me7}
\end{equation}
Substituting \eqref{me7} into \eqref{me6}, we have
\begin{align}
\phi(x^*,x_{n+1})\leq& \phi(x^*,x_n) - \phi(y_n,x_n) + 2\langle y_n-x^*, Jy_n-Jx_n\rangle - 2 \lambda_n\langle Ay_n-Ax_n,y_n-x^*\rangle \nonumber\\
&+  2\kappa^2\|\lambda_n(Ay_n-Ax_n)\|^2\nonumber\\
=&  \phi(x^*,x_n) - \phi(y_n,x_n)+  2\kappa^2\|\lambda_n(Ay_n-Ax_n)\|^2\nonumber\\
&-2\langle Jx_n-Jy_n-\lambda_n(Ax_n-Ay_n),y_n-x^*\rangle.\label{meee8}
\end{align}
Since $y_n= (J+\lambda_nB)^{-1} JoJ^{-1}(Jx_n - \lambda_nAx_n)$, we have
$Jx_n - \lambda_nAx_n \in (J+\lambda_nB)y_n$. Using the fact that $B$ is maximal monotone, then there exists $v_n \in By_n$ such that
 $Jx_n - \lambda_nAx_n=Jy_n+\lambda_n v_n$. Therefore
 \begin{equation}\label{me9}
v_n=\frac{1}{\lambda_n}(Jx_n-Jy_n-\lambda_n Ax_n).
\end{equation}
On the other hand, we know that $0 \in (Ax^*+Bx^*)$ and $Ay_n +v_n \in (A+B)y_n$. Since $A+B$ is maximal monotone, we obtain
\begin{eqnarray}\label{me10}
\langle Ay_n +v_n,y_n-x^*\rangle \geq 0.
\end{eqnarray}
Putting \eqref{me9} into \eqref{me10}, we get
\begin{eqnarray}\label{me11}
\langle Jx_n-Jy_n-\lambda_n(Ax_n-Ay_n),y_n-x^*\rangle \geq 0.
\end{eqnarray}
Now, using \eqref{me11} in \eqref{meee8}, we get
\begin{align}
\phi(x^*,x_{n+1})\leq& \phi(x^*,x_n) - \phi(y_n,x_n)+ 2\kappa^2\|\lambda_n(Ay_n-Ax_n)\|^2\nonumber\\
\leq&  \phi(x^*,x_n) - \phi(y_n,x_n)+  2\kappa^2\lambda_n^2L^2\mu\phi(y_n,x_n)\nonumber\\
=&\phi(x^*,x_n) -(1-2\kappa^2\lambda_n^2L^2\mu)\phi(y_n,x_n).\label{me8}
\end{align}
Using Assumption~\ref{Ass:Parameters}, we get
\begin{eqnarray}\label{ert}
\phi(x^*,x_{n+1}) \leq \phi(x^*, x_n),
\end{eqnarray}
which shows that $\lim\phi(x^*, x_n)$ exists and hence, $\{\phi(x^*, x_n)\}$ is bounded.
Therefore $\{x_n\}$ is bounded.
\end{proof}

\begin{dfn}
The duality mapping $J$ is weakly sequentially continuous if, for any sequence $\{x_n\}\subset E$ such that $x_n\rightharpoonup x$ as $n\rightarrow \infty$, then $Jx_n\rightharpoonup^* Jx$ as $n\rightarrow \infty$. It is known that the normalized duality map on $\ell_p$ spaces, $1 < p <\infty$, is weakly sequentially continuous.
\end{dfn}

\noindent
We now obtain the weak convergence result of Algorithm \ref{Alg:AlgL} in the next theorem.

\begin{thm}\label{th1}
Let Assumptions~\ref{Ass:VI} and \ref{Ass:Parameters} hold. Assume that $J$ is weakly sequentially continuous on $E$ and let the sequence $ \{x_n\}_{n=1}^\infty $ be generated by Algorithm~\ref{Alg:AlgL}. Then $\{x_n\}$ converges weakly to $z\in (A+B)^{-1}(0)$. Moreover, $z:=\underset{n\rightarrow \infty}\lim \Pi_{(A+B)^{-1}(0)}(x_n)$.
\end{thm}

\begin{proof}
Let $x^*\in (A+B)^{-1}(0)$. From \eqref{me8}, we have
\begin{eqnarray}\label{me12}
0&<& [1-2\kappa^2b^2L^2\mu]\phi(y_n, x_n)\leq [1-2\kappa^2\lambda_n^2L^2\mu]\phi(y_n, x_n)\nonumber \\
&\leq& \phi(x^*, x_n)-\phi(x^*,x_{n+1}).
\end{eqnarray}
Since $\lim_{n\rightarrow \infty}\phi(x^*, x_n)$ exists, we obtain from \eqref{me12} that
$$
\underset{n\rightarrow \infty}\lim \phi(y_n, x_n)=0.
$$
\noindent Applying Lemma \ref{aoy}, we get
$$
\underset{n\rightarrow \infty}\lim \|x_n-y_n\|=0.
$$
\noindent Since $E$ is uniformly smooth, the duality mapping $J$ is uniformly norm-to-norm continuous on each bounded subset
of $E$. Hence, we have
$$
\underset{n\rightarrow \infty}\lim \|Jx_n-Jy_n\|=0.
$$
\noindent
Since $\{x_n\}$ is bounded by Lemma \ref{t31}, there exists a subsequence $\{x_{n_i}\}$ of $\{x_n\}$ and $z\in C$ such that $x_{n_i}\rightharpoonup z$. Since $\underset{n\rightarrow \infty}\lim \|x_n-y_n\|=0$, it follows that $x_{{n_i}+1}\rightharpoonup z$. We now show that $z\in (A+B)^{-1}(0)$.\\

\noindent Suppose $(v,u) \in \textrm{Graph}(A+B)$. This implies that $Ju-Av \in Bv$. Furthermore, we obtain from
$y_{n_i}= (J+\lambda_{n_i}B)^{-1} JoJ^{-1}(Jx_{n_i} - \lambda_{n_i}Ax_{n_i})$ that
$$
(J-\lambda_{n_i}A)x_{n_i} \in (J+\lambda_{n_i} B)y_{n_i},
$$
\noindent and thus
$$
\frac{1}{\lambda_{n_i}}(Jx_{n_i}-Jy_{n_i}-\lambda_{n_i} Ax_{n_i}) \in By_{n_i}.
$$
\noindent Using the fact that $B$ is maximal monotone, we obtain
$$
\langle v-y_{n_i},Ju-Av-\frac{1}{\lambda_{n_i}}(Jx_{n_i}-Jy_{n_i}-\lambda_{n_i} Ax_{n_i})\rangle \geq 0.
$$
\noindent Therefore,
\begin{eqnarray*}
\langle v-y_{n_i},Ju\rangle &\geq& \langle v-y_{n_i},Av+\frac{1}{\lambda_{n_i}}(Jx_{n_i}-Jy_{n_i}-\lambda_{n_i} Ax_{n_i})\rangle \nonumber \\
&=& \langle v-y_{n_i}, Av-Ax_{n_i}\rangle +\langle v-y_{n_i}, \frac{1}{\lambda_{n_i}}(Jx_{n_i}-Jy_{n_i})\rangle \nonumber \\
&=& \langle v-y_{n_i}, Av-Ay_{n_i}\rangle + \langle v-y_{n_i}, Ay_{n_i}-Ax_{n_i}\rangle \nonumber \\
&&+\langle v-y_{n_i}, \frac{1}{\lambda_{n_i}}(Jx_{n_i}-Jy_{n_i})\rangle \nonumber \\
&\geq & \langle v-y_{n_i}, Ay_{n_i}-Ax_{n_i}\rangle+\langle v-y_{n_i}, \frac{1}{\lambda_{n_i}}(Jx_{n_i}-Jy_{n_i})\rangle.
\end{eqnarray*}
By the fact that $\underset{n\rightarrow \infty}\lim \|x_n-y_n\|=0$ and $A$ is Lipschitz continuous, we obtain
$\underset{n\rightarrow \infty}\lim \|Ax_n-Ay_n\|=0$. Consequently, we obtain that
$$
\langle v-z,Ju\rangle \geq 0.
$$
\noindent By the maximal monotonicity of $A+B$, we have $0\in (A+B)z$. Hence, $z \in (A+B)^{-1}(0)$.\\

\noindent Let $u_n:=\Pi_{(A+B)^{-1}(0)} (x_n)$. By \eqref{ert} and Lemma \ref{lm28}, we have that $\{u_n\}$ is a Cauchy sequence. Since $(A+B)^{-1}(0)$ is closed, we have that $\{u_n\}$ converges strongly to $ w \in (A+B)^{-1}(0)$. By the uniform smoothness of $E$, we also have
$\underset{n\rightarrow \infty}\lim \|Ju_n-Jw\|=0$. We then show that $z=w$. Using Lemma \ref{new2} (i), $u_n=\Pi_{(A+B)^{-1}(0)} (x_n)$ and $ z \in (A+B)^{-1}(0)$, we have
$$
\langle z-u_n,Ju_n-Jx_n\rangle \geq 0,~~\forall n \geq 1.
$$
\noindent Therefore,
\begin{eqnarray*}
\langle z-w,Jx_n-Ju_n\rangle&=& \langle z-u_n,Jx_n-Ju_n\rangle+\langle u_n-w,Jx_n-Ju_n\rangle\\
&\leq&\|u_n-w\|\|Jx_n-Ju_n\|\leq M\|u_n-w\|,~~\forall n \geq 1,
\end{eqnarray*}
where $M:=\underset{n \geq 1}\sup\|Jx_n-Ju_n\|$. Using $n=n_{i}$ in $\underset{n\rightarrow \infty}\lim\|u_n-w\|=0,
\underset{n\rightarrow \infty}\lim\|Ju_n-Jw\|=0$ and the weakly sequential continuity of $J$, we obtain
$$
\langle z-w,Jz-Jw\rangle \leq 0
$$
\noindent as $i\rightarrow \infty$. Therefore, $ \langle z-w,Jz-Jw\rangle = 0$. Since $E$ is strictly convex, we have $z = w$. Therefore, the sequence $\{x_n\}$ converges weakly to $z=\underset{n\rightarrow \infty}\lim \Pi_(A+B)^{-1}(0) (x_n)$. This completes the proof.
\end{proof}

\noindent It is easy to see from Algorithm \ref{Alg:AlgL} above and Lemma \ref{babu} that $x_n=y_n$ if and only if  $x_n \in (A+B)^{-1}(0)$. Also, we have already established that $\|x_n-y_n\|\rightarrow 0$ holds when $(A+B)^{-1}(0)\neq \emptyset$. Therefore, using the
$\|x_n-y_n\|$ as a measure of convergence rate, we obtain the following non asymptotic rate of convergence of our proposed Algorithm \ref{Alg:AlgL}.

\begin{thm}
Let Assumptions~\ref{Ass:VI} and \ref{Ass:Parameters} hold. Let the sequence $ \{x_n\}_{n=1}^\infty $ be generated by Algorithm~\ref{Alg:AlgL}. Then $\min_{1\leq k\leq n}\|x_k-y_k\|=O(1/\sqrt{n})$.
\end{thm}

\begin{proof}
We obtain from \eqref{me8} that
\begin{eqnarray*}
\phi(x^*,x_{n+1})\leq \phi(x^*,x_n)-(1-2\kappa^2\lambda_n^2L^2\mu)\phi(y_n,x_n).
\end{eqnarray*}
Hence, we have from Lemma \ref{aoy} that
\begin{eqnarray*}
\frac{1}{\mu}(1-2\kappa^2\lambda_n^2L^2\mu)\|x_n-y_n\|^2&\leq& (1-2\kappa^2\lambda_n^2L^2\mu)\phi(y_n,x_n)\\
&\leq& \phi(x^*,x_n)-\phi(x^*,x_{n+1}).
\end{eqnarray*}
By Assumption~\ref{Ass:Parameters}, we get
\begin{eqnarray*}
\sum_{k=1}^n\|x_k-y_k\|^2\leq \frac{\mu}{(1-2\kappa^2\lambda_n^2L^2\mu)}\phi(x^*,x_1).
\end{eqnarray*}
Therefore,
$$\min_{1\leq k\leq n}\|x_k-y_k\|^2\leq \frac{\mu}{n(1-2\kappa^2\lambda_n^2L^2\mu)}\phi(x^*,x_1).$$
\noindent
This implies that
$$\min_{1\leq k\leq n}\|x_k-y_k\|=O(1/\sqrt{n}).$$

\end{proof}

\noindent Next, we propose another iterative method such that the sequence of step-sizes does not depend on the Lipschitz constant of monotone operator $A$ in problem \eqref{bami2}.

\begin{alg}\label{Alg:AlgL1}$\left.  {}\right.$\newline
\textbf{Step 0}: Let Assumption~\ref{Ass:VI} hold. Given $\gamma >0, l \in (0,1)$ and $\theta \in (0,\frac{1}{\sqrt{2\mu}\kappa})$. Let $ x_1 \in E $ be a given starting point. Set $ n := 1 $.\newline

\noindent\textbf{Step 1}: Compute $y_n:= J_{\lambda_n}^BJ^{-1}(Jx_n - \lambda_nAx_n)$, where
$\lambda_n$ is chosen to be the largest
$$
\lambda \in \{\gamma,\gamma l,\gamma l^2,\ldots\}
$$
\noindent satisfying
\begin{eqnarray}\label{seti}
\lambda\|Ax_n-Ay_n\| \leq \theta\|x_n-y_n\|.
\end{eqnarray}
If $x_n-y_n=0$: STOP.\newline

\noindent\textbf{Step 2}: Compute
\begin{eqnarray}\label{e41}
 x_{n+1} = J^{-1}[Jy_n - \lambda_n(Ay_n-Ax_n)].
\end{eqnarray}

\noindent\textbf{Step 3}: Set $ n \leftarrow n+1 $, and go to \textbf{Step 1}.
\end{alg}

\noindent Before we establish the weak convergence analysis of Algorithm \ref{Alg:AlgL1}, we first show that the line search rule given in \eqref{seti} is well-defined in this lemma.

\begin{lem}
The line search rule \eqref{seti} in Algorithm \ref{Alg:AlgL1} is well-defined and
$$
\min\Big\{\gamma,\frac{\theta l}{L} \Big\}\leq \lambda_n \leq \gamma.
$$

\end{lem}

\begin{proof}
Using the Lipschitz continuity of $A$ on $E$, we obtain
$$
\|Ax_n-A(J_{\lambda_n}^BJ^{-1}(Jx_n - \lambda_nAx_n))\| \leq L\|x_n-J_{\lambda_n}^BJ^{-1}(Jx_n - \lambda_nAx_n)\|.
$$
\noindent This implies that
$$
\frac{\theta}{L}\|Ax_n-A(J_{\lambda_n}^BJ^{-1}(Jx_n - \lambda_nAx_n))\| \leq \theta\|x_n-J_{\lambda_n}^BJ^{-1}(Jx_n - \lambda_nAx_n)\|.
$$
\noindent Therefore, \eqref{seti} holds whenever $\lambda_n \leq \frac{\theta}{L}$. Hence, $ \lambda_n$ is well-defined.\\

\noindent From the way $ \lambda_n$ is chosen, we can clearly see that $ \lambda_n \leq \gamma$. Now, suppose $ \lambda_n=\gamma$, then \eqref{seti} is satisfied and the lemma is proved. Suppose $ \lambda_n<\gamma$. Then $\frac{ \lambda_n}{l}$ violates \eqref{seti} and we get
\begin{eqnarray*}
L\|x_n-J_{\lambda_n}^BJ^{-1}(Jx_n - \lambda_nAx_n)\| &\geq & \|Ax_n-A(J_{\lambda_n}^BJ^{-1}(Jx_n - \lambda_nAx_n))\|\\
&>& \frac{\theta}{\frac{\lambda_n}{l}} \|x_n-J_{\lambda_n}^BJ^{-1}(Jx_n - \lambda_nAx_n)\|.
\end{eqnarray*}
This implies that $\lambda_n > \frac{\theta l}{L}$.  This completes the proof.
\end{proof}

\noindent We now give a weak convergence result using Algorithm \ref{Alg:AlgL1} in the next theorem.

\begin{thm}\label{th2}
Let Assumptions~\ref{Ass:VI}. Assume that $J$ is weakly sequentially continuous on $E$ and let the sequence $ \{x_n\}_{n=1}^\infty $ be generated by Algorithm~\ref{Alg:AlgL1}. Then $\{x_n\}$ converges weakly to $z\in (A+B)^{-1}(0)$. Moreover, $z:=\underset{n\rightarrow \infty}\lim \Pi_{(A+B)^{-1}(0)}(x_n)$.
\end{thm}

\begin{proof}
Using the same line of arguments as in the proof of Lemma \ref{t31}, we can obtain from \eqref{me8} that
\begin{eqnarray}\label{hm}
\phi(x^*,x_{n+1})&\leq& \phi(x^*,x_n) - \phi(y_n,x_n)+ 2\kappa^2\|\lambda_n(Ay_n-Ax_n)\|^2\nonumber\\
&\leq&  \phi(x^*, x_n) -\phi(y_n, x_n) +2\kappa^2\theta^2\|y_n-x_n\|^2 \nonumber\\
&\leq&  \phi(x^*,x_n) - \phi(y_n,x_n)+  2\kappa^2\theta^2\mu\phi(y_n,x_n)\nonumber\\
&=&\phi(x^*,x_n) -(1-2\kappa^2\theta^2\mu)\phi(y_n,x_n).
\end{eqnarray}
Since $\theta^2 <\frac{1}{2\kappa^2\mu}$, we get
\begin{eqnarray}\label{ert23}
\phi(x^*,x_{n+1}) \leq \phi(x^*, x_n),
\end{eqnarray}
which shows that $\lim\phi(x^*, x_n)$ exists and hence, $\{\phi(x^*, x_n)\}$ is bounded.
Therefore $\{x_n\}$ is bounded. The rest of the proof follows by using the same arguments as in the proof of Theorem \ref{th1}. The completes the proof.
\end{proof}

\noindent Finally, we give a modification of Algorithm \ref{Alg:AlgL} and consequently obtain the strong convergence analysis below.

\begin{alg}\label{Alg:AlgL7}$\left.  {}\right.$\newline
\textbf{Step 0}: Let Assumptions~\ref{Ass:VI} and \ref{Ass:Parameters} hold. Suppose that $\{\alpha_n\}$ is a real sequence in (0,1) and let $ x_1 \in E $ be a given starting point. Set $ n := 1 $.\newline

\noindent\textbf{Step 1}: Compute $y_n:= J_{\lambda_n}^BJ^{-1}(Jx_n - \lambda_nAx_n)$. If $x_n-y_n=0$: STOP.\newline

\noindent\textbf{Step 2}: Compute
\begin{eqnarray}\label{e44}
 w_n = J^{-1}[Jy_n - \lambda_n(Ay_n-Ax_n)]
\end{eqnarray}
and
\begin{eqnarray}\label{e45}
 x_{n+1} = J^{-1}[\alpha_nJx_1+(1-\alpha_n)Jw_n].
\end{eqnarray}

\noindent\textbf{Step 3}: Set $ n \leftarrow n+1 $, and go to \textbf{Step 1}.
\end{alg}

\begin{thm}
Let Assumptions~\ref{Ass:VI} and \ref{Ass:Parameters} hold. Suppose that $\lim_{n \to \infty} \alpha_n = 0 $ and $ \sum_{n = 1}^{\infty} \alpha_n = \infty $. Let the sequence $ \{x_n\}_{n=1}^\infty $ be generated by Algorithm~\ref{Alg:AlgL7}. Then $\{x_n\}$ converges strongly to $z=\Pi_{(A+B)^{-1}(0)}(x_1)$.
\end{thm}

\begin{proof}
By Lemma \ref{t31}, we have that $\{x_n\}$ is bounded. Furthermore, using Lemma \ref{new2} with \eqref{e44} and \eqref{e45}, we have
\begin{eqnarray}\label{seyi}
\phi(z,x_{n+1})&=&\phi(z,J^{-1}(\alpha_n Jx_1+(1-\alpha_n)Jw_n))\nonumber\\
&=&V(z,\alpha_n Jx_1+(1-\alpha_n)Jw_n))\nonumber\\
&\leq&V(z,\alpha_n Jx_1+(1-\alpha_n)Jw_n-\alpha_n(Jx_1-Jz))\nonumber\\
&&+2\alpha_n\langle Jx_1-Jz,x_{n+1}-z\rangle \nonumber\\
&=&V(z,\alpha_n Jz+(1-\alpha_n)Jw_n)+2\alpha_n\langle Jx_1-Jz,x_{n+1}-z\rangle\nonumber\\
&\leq& \alpha_n V(z,Jz)+(1-\alpha_n)V(z,Jw_n)+2\alpha_n\langle Jx_1-Jz,x_{n+1}-z\rangle\nonumber\\
&=&(1-\alpha_n)V(z,Jw_n)+2\alpha_n\langle Jx_1-Jz,x_{n+1}-z\rangle\nonumber\\
&\leq& (1-\alpha_n)V(z,Jx_n)+2\alpha_n\langle Jx_1-Jz,x_{n+1}-z\rangle\nonumber\\
&=&(1-\alpha_n)\phi(z,x_n)+2\alpha_n\langle Jx_1-Jz,x_{n+1}-z\rangle.
\end{eqnarray}
Set $a_n:=\phi(x_n,z)$ and divide the rest of the proof into two parts as follows.

{\em Case 1:} Suppose that there exists $n_0 \in \mathbb{N}$ such
that $\{\phi(z,x_n)\}_{n=n_0}^{\infty}$ is non-increasing. Then
$\{\phi(z,x_n)\}_{n=1}^{\infty}$  converges, and we therefore obtain
\begin{equation}\label{Eq:diffconv}
   a_n-a_{n+1}\rightarrow 0,~~n\rightarrow \infty.
\end{equation}
Using (\ref{ert}) in (\ref{e45}), we have
\begin{eqnarray}\label{fra}
V(z,Jx_{n+1}) &\leq& \alpha_n V(z,Jx_1)+(1-\alpha_n)V(z,Jw_n) \nonumber \\
&\leq& \alpha_n V(Jx_1,z)+(1-\alpha_n)V(Jx_n,z) \nonumber \\
&&-(1-\alpha_n)[1-2\kappa^2\theta^2\mu]V(y_n,Jx_n).
\end{eqnarray}
This implies from (\ref{fra}) that
$$(1-\alpha_n)[1-2\kappa^2\theta^2\mu]V(y_n,Jx_n)
\leq V(Jx_n,z)-V(Jx_{n+1},z)+ \alpha_nM_1,$$
 for some $M_1>0$. Thus,
$$(1-\alpha_n)[1-2\kappa^2\theta^2\mu]\phi(y_n,x_n)
\rightarrow 0,~~n\rightarrow \infty.$$
Hence,
$$
\phi(y_n,x_n)\rightarrow 0,~~n\rightarrow \infty.
$$
Consequently, $ \|x_n-y_n\|\rightarrow 0,~~n\rightarrow \infty.$
By \eqref{e44}, we get
\begin{eqnarray*}
\|Jw_n-Jy_n\|&=&\lambda_n\|Ay_n-Ax_n\|\\
&\leq& b\|Ay_n-Ax_n\|\rightarrow 0,~~n\rightarrow \infty.
\end{eqnarray*}
Therefore, $\|w_n-y_n\| \rightarrow 0,~~n\rightarrow \infty.$
Moreover, we obtain from (\ref{e45}) that
\begin{eqnarray}\label{fra2}
\|Jx_{n+1}-Jw_n\|&=&\alpha_n\|Jx_1-Jw_n\|  \leq \alpha_nM_2\rightarrow 0,~~n\rightarrow \infty,
\end{eqnarray}
for some $M_2>0$. Since $J^{-1}$ is norm-to-norm uniformly continuous on bounded subsets of $E^*$, we have that
$$\|x_{n+1}-w_n\|\rightarrow 0,~~n\rightarrow \infty.$$
 Now,
$$\|x_{n+1}-x_n\|\leq \|x_{n+1}-w_n\|+\|w_n-y_n\|+\|y_n-x_n\| \rightarrow 0,~~n\rightarrow \infty.$$
Since $\{x_n\}$ is a bounded sunset of $E$, we can choose a subsequence $\{x_{n_k}\}$ of $\{x_n\}$ such that
$x_{n_k}\rightharpoonup p \in E$ and
\begin{eqnarray*}
\limsup_{n\rightarrow \infty}  \langle Jx_1-Jz,x_n-z\rangle \leq 2 \lim_{k\rightarrow \infty} \langle Jx_1-Jz,x_{n_k}-z\rangle.
\end{eqnarray*}
Since $z=\Pi_C x_1$, we get
\begin{eqnarray}\label{j13}
\limsup_{n\rightarrow \infty}  \langle Jx_1-Jz,x_n-z\rangle &\leq& 2 \lim_{k\rightarrow \infty} \langle Jx_1-Jz,x_{n_k}-z\rangle \nonumber \\
&=& 2\langle Jx_1-Jz,p-z\rangle \leq 0.
\end{eqnarray}
This implies that
$$
\limsup_{n\rightarrow \infty} \langle Jx_1-Jz,x_n-z\rangle  \leq 0.
$$
Using Lemma \ref{lm23} and (\ref{j13}) in \eqref{seyi}, we obtain
$\lim_{n \rightarrow \infty}\phi(z,x_n)=0.$  Thus, $x_n\rightarrow z$, $n \rightarrow \infty$.

{\em Case 2:} Suppose that there exists a subsequence $\{x_{n_j}\}$ of $\{x_n\}$ such that
$$\phi(z,x_{m_j}) <\phi(z,x_{{m_j}+1}),~~\forall j\in \mathbb{N}.$$
From Lemma \ref{mm}, there exists a nondecreasing sequence $\{n_k\}$ of $\mathbb{N}$ such that
 $\lim_{k\rightarrow \infty}\lim n_k=\infty$ and the following inequalities hold for all $k \in \mathbb{N}$:
\begin{eqnarray}\label{opr}
\phi(z,x_{n_k}) \leq \phi(z,x_{{n_k}+1})~~{\rm and}~~ \phi(z,x_k) \leq \phi(z,x_{{n_k}+1}).
\end{eqnarray}
 Observe that
\begin{eqnarray*}
\phi(z,x_{n_k}) &\leq& \phi(z,x_{{n_k}+1}) \leq \alpha_{n_k}\phi(z,x_1)+(1-\alpha_{n_k})\phi(z,w_{n_k})\\
&\leq& \alpha_{n_k}\phi(z,x_1)+(1-\alpha_{n_k})\phi(z,x_{n_k}).
\end{eqnarray*}
Since $\lim_{n\rightarrow \infty}\alpha_n=0$, we get
$$\phi(z,x_{{n_k}+1})-\phi(z,x_{n_k})\rightarrow 0,~~k\rightarrow \infty.$$
Since $\{x_{n_k}\}$ is bounded, there exists a subsequence of $\{x_{n_k}\}$ still denoted by
$\{x_{n_k}\}$ which converges weakly to $ p\in E$. Repeating the same arguments as in \textit{Case 1} above, we can show that
$$\|x_{n_k}-y_{n_k}\|\rightarrow 0,~~k\rightarrow \infty,\|y_{n_k}-w_{n_k}\|\rightarrow 0,~~k\rightarrow \infty
 ~~{\rm and }~~\|x_{{n_k}+1}-x_{n_k}\|\rightarrow 0,~~k\rightarrow \infty.$$
Similarly, we can conclude that
\begin{eqnarray}\label{mm1}
\limsup_{k\rightarrow \infty} \langle x_{{n_k}+1}-z, Jx_1-Jz\rangle&=&
\limsup_{k\rightarrow \infty} \langle x_{n_k}-z, Jx_1-Jz\rangle \leq 0.
\end{eqnarray}
It then follows from \eqref{seyi} and (\ref{opr}) that
\begin{eqnarray*}
\phi(z,x_{{n_k}+1}) &\leq& (1-\alpha_{n_k})\phi(z,x_{n_k})+\alpha_{n_k}\langle x_{{n_k}+1}-z, Jx_1-Jz\rangle\\
&\leq& (1-\alpha_{n_k})\phi(z,x_{{n_k}+1})+\alpha_{n_k}\langle x_{{n_k}+1}-z, Jx_1-Jz\rangle.
\end{eqnarray*}
Since  $\alpha_{n_k}>0$, we get
$$ \phi(z,x_{n_k}) \leq \phi(z,x_{{n_k}+1}) \leq \langle x_{{n_k}+1}-z, Jx_1-Jz\rangle.$$
By (\ref{mm1}), we have that
$$\limsup_{k\rightarrow \infty} \phi(z,x_{n_k}) \leq \limsup_{k\rightarrow \infty} \langle x_{{n_k}+1}-z, Jx_1-Jz\rangle.$$
Therefore, $x_k\rightarrow z,~~k\rightarrow \infty.$
This concludes the proof.

\end{proof}

\begin{rem}
Our proposed Algorithms \ref{Alg:AlgL} and \ref{Alg:AlgL7} are more applicable than the proposed methods in \cite{Briceno,Cho,JiaoWang,Lin,LopezG,TakahashiW1,TakahashiW2,TakahashiW3,ShehuY,WangWang} even in Hilbert spaces. The methods proposed in \cite{Cho,JiaoWang,Lin,LopezG,TakahashiW1,TakahashiW2,TakahashiW3,ShehuY,WangWang} are only applicable for solving problem \eqref{bami2} in the case when $B$ is maximal monotone and $A$ is inverse-strongly monotone (co-coercive) operator in real Hilbert spaces. Our Algorithms \ref{Alg:AlgL} and \ref{Alg:AlgL7} are applicable for the case when $B$ is maximal monotone and $A$ is monotone operator even in 2-uniformy convex and uniformly smooth Banach spaces (e.g., $L_p, 1<p\leq 2$). Our results in this paper also complement the results of \cite{CombettesNguyen,IusemAlfredo}.

\end{rem}

\section{Application}\label{appl}
\noindent
In this section, we apply our results to the minimization of composite objective function of the type
\begin{eqnarray}\label{appl1}
\min_{x \in E} f(x)+g(x),
\end{eqnarray}
where $f:E\rightarrow \mathbb{R}\cup \{ +\infty\}$ is proper, convex and lower semi-continuous functional and
$g:E\rightarrow \mathbb{R}$ is convex functional.\\

\noindent
Many optimization problems from image processing \cite{Bredies}, statistical regression, machine learning (see, e.g., \cite{WangYamin} and the references contained therein), etc can be adapted into the form of \eqref{appl1}.
In this setting, we assume that $g$ represents the "smooth part" of the functional where $f$ is assumed to be non-smooth. Specifically, we assume that $g$ is G$\hat{a}$teaux-differentiable with derivative $\nabla g$ which is Lipschitz-continuous with constant $L$. Then by
\cite[thm.\ 3.13]{Peypouquetbook}, we have
$$
\langle \nabla g (x)-\nabla g(y), x-y\rangle \geq \frac{1}{L}\|\nabla g(x)-\nabla g(y)\|^2, ~~\forall x, y \in E.
$$
\noindent
Therefore, $\nabla g$ is monotone and Lipschitz continuous with Lipschitz constant $L$. Observe that problem \eqref{appl1} is equivalent to
find $ \in E$ such that
\begin{eqnarray}\label{appl2}
0\in \partial f(x)+ \nabla g(x).
\end{eqnarray}
Then problem \eqref{appl2} is a special case of inclusion problem \eqref{bami2} with $A:=\nabla g$ and $B:=\partial f$.

Next, we obtain the resolvent of  $\partial f$. Let us fix $r>0$ and $z \in E$. Suppose $J_r^{\partial f}$ is the resolvent of
$\partial f$. Then
$$
Jz \in J (J_r^{\partial f})+r \partial f(J_r^{\partial f}).
$$
\noindent Hence we obtain
  \begin{eqnarray*}
  0\in \partial f(J_r^{\partial f})+\frac{1}{r}J (J_r^{\partial f})-\frac{1}{r}Jz
 =\partial \Big(f+\frac{1}{2r}\|.\|^2-\frac{1}{r}Jz \Big)J_r^{\partial f}.
  \end{eqnarray*}
Therefore,
$$
J_r^{\partial f}(z)={\rm argmin}_{y \in E}\Big\{f(y)+\frac{1}{2r}\|y\|^2-\frac{1}{r}\langle y,Jz\rangle \Big\}.
$$
\noindent
We can then write $y_n$ in Algorithm \ref{Alg:AlgL} as
$$
y_n={\rm argmin}_{y \in E}\Big\{f(y)+\frac{1}{2\lambda_n}\|y\|^2-\frac{1}{\lambda_n}\langle y,Jx_n - \lambda_n \nabla g(x_n)\rangle \Big\}.
$$
\noindent We obtain the following weak and strong convergence results for problem \eqref{appl1}.

\begin{thm}\label{appl3}
Let $E$ be a real 2-uniformly convex Banach space which is also uniformly smooth and the solution set $S$ of  problem \eqref{appl1} be nonempty.
Suppose $ \{ \lambda_n \}_{n=1}^\infty $ satisfies the condition $0<a\leq\lambda_n\leq b<\displaystyle\frac{1}{\sqrt{2\mu}\kappa L}$.
Assume that $J$ is weakly sequentially continuous on $E$ and let the sequence $ \{x_n\}_{n=1}^\infty $ be generated by
\begin{eqnarray}\label{appl4}
\left\{  \begin{array}{llll}
   & x_1 \in E,\\
   & y_n={\rm argmin}_{y \in E}\Big\{f(y)+\frac{1}{2\lambda_n}\|y\|^2-\frac{1}{\lambda_n}\langle y,Jx_n - \lambda_n \nabla g(x_n)\rangle \Big\}\\
   & x_{n+1} = J^{-1}[Jy_n - \lambda_n(\nabla g(y_n)-\nabla g(x_n))],~~n \geq 1.
   \end{array}
   \right.
\end{eqnarray}
Then $\{x_n\}$ converges weakly to $z\in S$. Moreover, $z:=\underset{n\rightarrow \infty}\lim \Pi_S(x_n)$.
\end{thm}

\begin{thm}\label{appl5}
Let $E$ be a real 2-uniformly convex Banach space which is also uniformly smooth and the solution set $S$ of  problem \eqref{appl1} be nonempty.
Suppose $ \{ \lambda_n \}_{n=1}^\infty $ satisfies the condition $0<a\leq\lambda_n\leq b<\displaystyle\frac{1}{\sqrt{2\mu}\kappa L}$.
Suppose that $\{\alpha_n\}$ is a real sequence in (0,1) with $\lim_{n \to \infty} \alpha_n = 0 $ and $ \sum_{n = 1}^{\infty} \alpha_n = \infty $. Let the sequence $ \{x_n\}_{n=1}^\infty $ be generated by
\begin{eqnarray}\label{appl6}
\left\{  \begin{array}{llll}
   & x_1 \in E,\\
   & y_n={\rm argmin}_{y \in E}\Big\{f(y)+\frac{1}{2\lambda_n}\|y\|^2-\frac{1}{\lambda_n}\langle y,Jx_n - \lambda_n \nabla g(x_n)\rangle \Big\}\\
   & w_n = J^{-1}[Jy_n - \lambda_n(\nabla g(y_n)-\nabla g(x_n))],\\
   &  x_{n+1} = J^{-1}[\alpha_nJx_1+(1-\alpha_n)Jw_n],~~n \geq 1.
   \end{array}
   \right.
\end{eqnarray}
Then $\{x_n\}$ converges strongly to $z=\Pi_S(x_1)$.
\end{thm}

\begin{rem}
\begin{itemize}
  \item Our result in Theorems \ref{appl3} and \ref{appl5} complement the results of Bredies \cite{Bredies,Guan}. Consequently, our results in Section \ref{Sec:Convergence} extend the results of Bredies \cite{Bredies,Guan} to inclusion problem \eqref{bami2}. In particular, we do not assume boundedness of $\{x_n\}$ (which was imposed on the results of \cite{Bredies,Guan}) in our results. Therefore, our result improves on the results of \cite{Bredies,Guan}.
  \item The minimization problem \eqref{appl1} in this section extends the problem studied in \cite{Beck,Combettes10,Nguyen,WangYamin} and other related papers from Hilbert spaces to Banach spaces.
 \end{itemize}
\end{rem}

\section{Conclusion}\label{Sec:Final}
\noindent
We study the Tseng-type algorithm for finding a solution to monotone inclusion problem involving a sum of maximal monotone and a Lipschitz continuous monotone mapping in 2-uniformly convex Banach space which is also uniformly smooth. We prove both weak and strong convergence of sequences of iterates to the solution of the inclusion problem under some appropriate conditions. Many results on monotone inclusion problems with single maximal monotone operator can be considered as special cases of the problem studied in this paper.  As far as we know, this is the first time an inclusion problem involving sum of maximal monotone and Lipschitz continuous monotone operators will be studied in Banach spaces. Therefore, the results of this paper open up many forthcoming results regarding the inclusion problem studied in this paper.  Our next project involves the following.
\begin{itemize}
  \item The results in this paper exclude $L_p$ spaces with $p > 2$. Therefore, extension of the results in this paper to a more general reflexive Banach space will be desired.
  \item How to effectively compute the duality mapping $J$ and the resolvent of maximal monotone mapping $B$ during implementations of our proposed algorithms will be considered further.
  \item The numerical implementations of problem \eqref{bami2} arising from signal processing, image reconstruction, etc will be studied;
  \item Other ways of implementation of the step-sizes $\lambda_n$ to give faster convergence of the proposed methods in this paper will be given.
\end{itemize}

\noindent\textbf{Acknowledgements} The project of the author has received funding from the European Research Council (ERC) under the European Union’s Seventh Framework Program (FP7 - 2007-2013) (Grant agreement No. 616160)


\begin{thebibliography}{111}


\bibitem{alber1}  Alber, Y. I. Metric and generalized projection operators in Banach spaces: properties and applications. Theory and applications of nonlinear operators of accretive and monotone type, 15-50, Lecture Notes in Pure and Appl. Math., 178, Dekker, New York, 1996.

%\bibitem{AlberSoviet} Alber, Ya. I.
%Recurrence relations and variational inequalities. (Russian)
%Dokl. Akad. Nauk SSSR 270 (1983), no. 1, 11-17.


%\bibitem{AlberNotik} Alber, Ya. I.; Notik, A. I.
%On the iterative method for variational inequalities with nonsmooth unbounded operators in Banach space.
%J. Math. Anal. Appl. 188 (1994), no. 3, 928-939.

\bibitem{AlberRyazantseva} Alber, Y.; Ryazantseva, I. Nonlinear ill-posed problems of monotone type. Springer, Dordrecht, 2006. xiv+410 pp. ISBN: 978-1-4020-4395-6; 1-4020-4395-3.

%\bibitem{AlMazrooei} {\rm Al-Mazrooei, A. E., Bin Dehaish, B. A., Latif, A., Yao, J. C.:} On general system of variational inequalities in Banach spaces. J. Nonlinear Convex Anal. 16 no. 4, 639–658 (2015).

\bibitem{Aoyama} Aoyama, K.; Kohsaka, F. Strongly relatively nonexpansive sequences generated by firmly nonexpansive-like mappings.
Fixed Point Theory Appl. 2014, 2014:95, 13 pp.


%\bibitem{Aubin}  Aubin, J.-P.; Ekeland, I. Applied nonlinear analysis. Pure and Applied Mathematics (New York). A Wiley-Interscience Publication. John Wiley \& Sons, Inc., New York, 1984. xi+518 pp. ISBN: 0-471-05998-6.

\bibitem{Avetisyan} Avetisyan, K.; Djordjevi\'{c}, O.; Pavlovi\'{c}, M.
Littlewood-Paley inequalities in uniformly convex and uniformly smooth Banach spaces.
J. Math. Anal. Appl. 336 (2007), no. 1, 31–43.


\bibitem{Ball} Ball, K.; Carlen, E. A.; Lieb, E. H.
Sharp uniform convexity and smoothness inequalities for trace norms.
Invent. Math. 115 (1994), no. 3, 463–482.


\bibitem{Barbu} Barbu, V. Nonlinear Semigroups and Differential Equations in Banach Spaces, Editura
Academiei R.S.R., Bucharest, 1976.

\bibitem{Beauzamy} Beauzamy, B. Introduction to Banach spaces and their geometry. Second edition. North-Holland Mathematics Studies, 68. Notas de Matemática [Mathematical Notes], 86. North-Holland Publishing Co., Amsterdam, 1985. xv+338 pp. ISBN: 0-444-87878-5.


\bibitem{Beck} Beck, A.; Teboulle, M. A fast iterative shrinkage-thresholding algorithm for linear inverse problems.
SIAM J. Imaging Sci. 2 (2009), no. 1, 183-202.


\bibitem{Bredies} Bredies, K.
A forward-backward splitting algorithm for the minimization of non-smooth convex functionals in Banach space.
Inverse Problems 25 (2009), no. 1, 015005, 20 pp.

\bibitem{Briceno}Brice\~no-Arias, L. M. Forward-partial inverse-forward splitting for solving monotone inclusions. J. Optim. Theory Appl. 166 (2015), no. 2, 391-413.


\bibitem{CChen} Chen, G. H.-G.; Rockafellar, R. T.
Convergence rates in forward-backward splitting.
SIAM J. Optim. 7 (1997), no. 2, 421–444.

\bibitem{Cho} Cho, S. Y.; Qin, X.; Wang, L.
Strong convergence of a splitting algorithm for treating monotone operators.
Fixed Point Theory Appl. 2014, 2014:94, 15 pp.

\bibitem{Cioranescu} Cioranescu, I. Geometry of Banach spaces, duality mappings and nonlinear problems. Mathematics and its Applications, 62. Kluwer Academic Publishers Group, Dordrecht, 1990. xiv+260 pp. ISBN: 0-7923-0910-3.


\bibitem{CombettesNguyen}
Combettes, P. L.; Nguyen, Q. V.
Solving composite monotone inclusions in reflexive Banach spaces by constructing best Bregman approximations from their Kuhn-Tucker set.
J. Convex Anal. 23 (2016), no. 2, 481-510.


\bibitem{Combettes10} Combettes, P.; Wajs, V. R.
Signal recovery by proximal forward-backward splitting.
Multiscale Model. Simul. 4 (2005), no. 4, 1168-1200.

\bibitem{Diestel} Diestel, J. Geometry of Banach spaces—selected topics. Lecture Notes in Mathematics, Vol. 485. Springer-Verlag, Berlin-New York, 1975, xi+282 pp.

\bibitem{Figiel} Figiel, T. On the moduli of convexity and smoothness. Studia Math. 56 (1976), no. 2, 121-155.


\bibitem{GibaliThong} Gibali, A.; Thong, D. V.
Tseng type methods for solving inclusion problems and its applications.
Calcolo 55 (2018), no. 4, 55:49.


\bibitem{Guan} Guan, W.-B., Song, W. The generalized forward-backward splitting method for the minimization of the sum of two functions in Banach spaces. Numer. Funct. Anal. Optim. 36 (2015), no. 7, 867-886.


\bibitem{Guler} G\"uler, O. On the convergence of the proximal point algorithm for convex minimization. SIAM J. Control Optim. 29 (1991), 403-419.


\bibitem{Iiduka} Iiduka, H.; Takahashi, W. Weak convergence of a projection algorithm for variational inequalities in a Banach space. J. Math. Anal. Appl. 339 (2008), no. 1, 668–679.

\bibitem{IusemAlfredo} Iusem, A. N.; Svaiter, B. F.
Splitting methods for finding zeroes of sums of maximal monotone operators in Banach spaces.
J. Nonlinear Convex Anal. 15 (2014), no. 2, 379-397.


\bibitem{JiaoWang} Jiao, H.; Wang, F.
On an iterative method for finding a zero to the sum of two maximal monotone operators.
J. Appl. Math. 2014, Art. ID 414031, 5 pp.


\bibitem{Kamimura}
Kamimura, S.; Kohsaka,  F.;  Takahashi, W.
Weak and Strong Convergence Theorems for Maximal Monotone Operators in a Banach Space,
Set-Valued Anal. 12 (2004), 417-429.

\bibitem{KamimuraTakahashi} Kamimura, S.; Takahashi, W. Strong convergence of a proximal-type algorithm in a Banach space. SIAM J. Optim. 13 (2002), no. 3, 938-945 (2003).

\bibitem{Kohsaka}
Kohsaka,  F.;  Takahashi, W.
Strong convergence of an iterative sequence for maximal monotone operators in a Banach space.
Abstr. Appl. Anal. 2004, no. 3, 239-249.

\bibitem{Lionspl}
Lions, P. L. Une m\'ethode it\'erative de r\'esolution d'une in\'equation variationnelle, Israel J. Math.
31 (1978), 204-208.

\bibitem{LionsMercier} Lions, P. L.;  Mercier, B. Splitting algorithms for the sum of two nonlinear operators, SIAM J. Numer. Anal., 16
(1979), 964-979.


\bibitem{Lin} Lin, L.-J.; Takahashi, W.
A general iterative method for hierarchical variational inequality problems in Hilbert spaces and applications.
Positivity 16 (2012), no. 3, 429-453.


\bibitem{LopezG} L\'{o}pez, G.; Mart\'{i}n-M\'{a}rquez, V.; Wang, F.; Xu, H.-K. Forward-backward splitting methods for accretive operators in Banach spaces. Abstr. Appl. Anal. 2012, Art. ID 109236, 25 pp.


\bibitem{mainge} Maing\'{e}, P.-E.
Strong convergence of projected subgradient methods for nonsmooth and nonstrictly convex minimization.
Set-Valued Anal. 16 (2008), no. 7-8, 899–912.


\bibitem{Martinet22} Martinet, B. R\'egularisation d'in\'equations variationnelles par approximations successives. (French) Rev. Française Informat. Recherche Opérationnelle 4 (1970), S\'er. R-3, 154–158.


\bibitem{MoudafiThera} Moudafi,  A.; Thera, M. Finding a zero of the sum of two maximal monotone operators, J. Optim. Theory Appl.,
94 (1997), 425–448.


\bibitem{Nguyen} Nguyen, T. P.; Pauwels, E.; Richard, E.; Suter, B. W. Extragradient method in optimization: convergence and complexity. J. Optim. Theory Appl. 176 (2018), no. 1, 137-162.

\bibitem{Passtygb}  Passty, G. B. Ergodic convergence to a zero of the sum of monotone operators in Hilbert spaces, J. Math. Anal.
Appl., 72 (1979), 383-390.

\bibitem{PeacemanRachford}  Peaceman, D. H.; Rachford, H. H.  The numerical solutions of parabolic and elliptic differential equations, J. Soc.
Indust. Appl. Math., 3 (1955), 28-41.


\bibitem{Peypouquetbook} Peypouquet, J. Convex optimization in normed spaces. Theory, methods and examples. With a foreword by Hedy Attouch. Springer Briefs in Optimization. Springer, Cham, 2015. xiv+124 pp. ISBN: 978-3-319-13709-4; 978-3-319-13710-0.

\bibitem{Reichbook}  Reich, S. A weak convergence theorem for the alternating method with Bregman distances, in: A.G. Kartsatos (Ed.), Theory and Applications of Nonlinear Operators of Accretive and Monotone Type, in: Lecture Notes Pure Appl. Math., vol. 178, Dekker, New York, 1996, pp. 313-318.


\bibitem{Rockafellar27} Rockafellar, R. T.  Monotone operators and the proximal point algorithm. SIAM J. Control. Optim. 14 (1976), 877-898.


\bibitem{RockafellarRT1}
Rockafellar, R. T. Characterization of the subdifferentials of convex functions. Pacific J. Math.
17 (1966), 497-510.

\bibitem{RockafellarRT2}
Rockafellar, R. T. On the maximal monotonicity of subdifferential mappings. Pacific J. Math.
33 (1970), 209-216.

%\bibitem{Reichs} Reich, S. A weak convergence theorem for the alternating method with Bregman distance,
%In: A. G. Kartsatos (ed.), Theory and Applications of Nonlinear Operators of Accretive and
%Monotone Type, Marcel Dekker, New York, 1996, pp. 313-318.


\bibitem{ShehuY} Shehu, Y.; Cai, G.
Strong convergence result of forward-backward splitting methods for accretive operators in Banach spaces with applications.
Rev. R. Acad. Cienc. Exactas F\'{i}s. Nat. Ser. A Math. RACSAM 112 (2018), no. 1, 71-87.


\bibitem{SolodovSvaiter} Solodov, M. V.; Svaiter, B. F. Forcing strong convergence of proximal point iterations in a
Hilbert space.  Math. Programing 87 (2000), 189-202.


\bibitem{TakahashiW1} Takahashi, S.; Takahashi, W.; Toyoda, M.
Strong convergence theorems for maximal monotone operators with nonlinear mappings in Hilbert spaces.
J. Optim. Theory Appl. 147 (2010), no. 1, 27-41.

\bibitem{TakahashiW2} Takahashi, W.; Wong, N.-C.; Yao, J.-C.
Two generalized strong convergence theorems of Halpern's type in Hilbert spaces and applications.
Taiwanese J. Math. 16 (2012), no. 3, 1151-1172.

\bibitem{TakahashiW3} Takahashi, W.
Strong convergence theorems for maximal and inverse-strongly monotone mappings in Hilbert spaces and applications.
J. Optim. Theory Appl. 157 (2013), no. 3, 781-802.

\bibitem{Taka} Takahashi, W. Nonlinear Functional Analysis, Yokohama Publishers, Yokohama 2000.

%\bibitem{Thong}  Thong, D. V.,  Hieu, D. V.  Weak and strong convergence theorems for variational inequality problems.
%In press: Numer. Algor. doi: 10.1007/s11075-017-0412-z.

\bibitem{Tseng}
Tseng, P.
A modified forward-backward splitting method for maximal monotone mappings.
SIAM J. Control Optim. 38 (2000), no. 2, 431–446.


\bibitem{WangWang}
Wang, Y.; Wang, F.
Strong convergence of the forward-backward splitting method with multiple parameters in Hilbert spaces.
Optimization 67 (2018), no. 4, 493–505.


\bibitem{WangYamin} Wang, Y.; Xu, H.-K. Strong convergence for the proximal-gradient method. J. Nonlinear Convex Anal. 15 (2014), no. 3, 581–593.


\bibitem{Xu} Xu, H. K. Inequalities in Banach spaces with applications. Nonlinear Anal. 16 (1991), no. 12, 1127–1138.

\bibitem{XU1} Xu, H. K.
Iterative algorithms for nonlinear operators.
J. London Math. Soc. (2) 66 (2002), no. 1, 240–256.

%\bibitem{Yao} Yao, Y., Aslam Noor, M., Inayat Noor, K., Liou, Y.-C., Yaqoob, H.: Modified extragradient methods for a system of variational inequalities in Banach spaces. Acta Appl. Math. 110 no. 3, 1211–1224 (2010).


%\bibitem{Zegeye} Zegeye, H., Shahzad, N.: Extragradient method for solutions of variational inequality problems in Banach spaces. Abstr. Appl. Anal. 2013, Art. ID 832548, 8 pp.


%\bibitem{Blum} E. Blum, W. Oettli, From optimization and variational inequalities to equilibrium problems, Math. Student, 63
%(1994), 123-145.

















  \end{thebibliography}
\end{document}